\documentclass[review]{elsarticle}

\usepackage[portrait, margin=1in]{geometry}
\usepackage{xcolor}
\usepackage{graphicx}
\usepackage{graphicx,epsfig}
\graphicspath{ {./images/} }
\usepackage{float}
\usepackage{enumitem}
\usepackage{amsmath}
\usepackage{amsthm}
\usepackage{multicol}
\usepackage{amssymb}
\usepackage{hyperref}
\usepackage{mathrsfs}
\usepackage{pstricks}
\usepackage{tikz}
\usepackage{tikz-cd}
\usepackage{amsfonts}
\usepackage{amssymb}
\usepackage{float}
\usepackage{url}
\usepackage{amsmath}
\usepackage{graphicx}
\usepackage{pstricks}
\usepackage{amsxtra}

\newcommand{\C}{\mathbb{C}}

\renewcommand{\Re}[1]{\operatorname{Re}\,{#1}}
\renewcommand{\Im}[1]{\operatorname{Im}\,{#1}}

\newcommand{\rk}[1]{\operatorname{rank}\,{#1}}
\newcommand{\card}[1]{\operatorname{card}\,{#1}}
\newcommand{\Ran}[1]{\text{Ran}\,{#1}}

\newtheorem{thm}{Theorem}[section]

\newtheorem{prop}[thm]{Proposition}
\newtheorem{defn}[thm]{Definition}

\theoremstyle{plain}

\theoremstyle{remark}

\theoremstyle{remark}

\theoremstyle{remark}

\definecolor{gold}{rgb}{0.85,.66,0}
\definecolor{cherry}{rgb}{0.9,.1,.2}
\definecolor{burgundy}{rgb}{0.8,.2,.2}
\definecolor{orangered}{rgb}{0.85,.3,0}
\definecolor{orange}{rgb}{0.85,.4,0}
\definecolor{olive}{rgb}{.45,.4,0}
\definecolor{lime}{rgb}{.6,.9,0}
\definecolor{green}{rgb}{.2,.7,0}
\definecolor{darkgreen}{rgb}{.1,.5,0}
\definecolor{grey}{rgb}{.4,.4,.2}
\definecolor{brown}{rgb}{.4,.2,.1}
\definecolor{blue}{rgb}{0,.0, .81}
\definecolor{bluepurple}{rgb}{.3, .0, .7}

\usepackage{natbib}
\begin{document}

\begin{frontmatter}
\title{{\bf The Role of Gr\"obner Bases in the Study \\ of Extremal Truncated Moment Problems}}

\author[iowa]{Ra\'ul E. Curto\corref{cor1}}
\ead{raul-curto@uiowa.edu}
\cortext[cor1]{Corresponding author}


\author[iowa]{Marc R. Moore}
\ead{marc-moore@uiowa.edu}

\address[iowa]{Department of Mathematics, The University of Iowa, Iowa City, Iowa, U.S.A.}

\begin{keyword} truncated complex moment problem; column relations; algebraic variety; Gr\"obner basis; harmonic polynomial.
 
\MSC[2020] {Primary 44A60; Secondary 47A57, 46G10.}
\end{keyword}

\begin{abstract}
In a 2014 paper, R.E. Curto and S. Yoo proved that a moment matrix $M(3)$ with specific harmonic polynomials as column relations admits a representing measure if and only if a condition at the level of moments holds. \ In this paper, we generalize the 2014 result to arbitrary moment matrices $M(k)$ ($k \in \mathbb{Z}_{+}$), with column relations given by general harmonic polynomials. \ We accomplish this by proving that the Gr\"obner basis for the ideal generated by a finite variety associated with the moment matrix provides all the necessary column relations for the matrix as well as a suitable condition on the moments, which is equivalent to the existence of a representing measure. 
\end{abstract}

\end{frontmatter}

\section{Introduction}

Using, as a building block, a 2014 result by R.E. Curto and S. Yoo \cite{CY14}, we consider a general class of truncated moment problems with column relations associated with a large class of polynomials in $\mathbb{C}[z,\bar{z}]$. \ In the process, we extend, in substantial and significant ways, the existing theory for sextic moment problems with singular moment matrices $M(3)$ and admitting harmonic polynomials as column relations. \ The new results yield necessary and sufficient conditions for the existence of finitely atomic representing measures for moment matrices $M(k)$ ($k \in \mathbb{Z}_{+}$), in terms of computable conditions at the level of the given moments. \ The class of polynomials that qualify for our approach is much larger than the cubic polynomials considered in \cite{CY14}. \ We accomplish this by observing that the Gr\"obner basis for the ideal generated by a finite variety associated with the moment matrix provides all the necessary column relations for the matrix as well as a suitable condition on the moments, which is equivalent to the existence of a representing measure (necessarily finitely atomic). \ Our results aim to continue to develop the natural interplay between operator theory, positive polynomials, and real algebraic geometry (see, e.g., \cite{Scheiderer24}, \cite{Schmudgen17}, \cite{Schweighofer05} and \cite{YooZalar24}) by emphasizing the role of Gr\"obner bases in determining the necessary and sufficient conditions for the existence of representing measures.

\section{Notation and Basic Terminology}
 
\subsection{Basic Theory of Truncated Moment Problems}

Let $\gamma \equiv \gamma^{(2k)} = \{ \gamma_{00}, \gamma_{01}, \gamma_{10}, \ldots, \gamma_{0,2k}, \gamma_{1,2k-1}, \ldots, \gamma_{2k-1,1}, \gamma_{2k} \}$ be a sequence of complex numbers with $\gamma_{00} > 0$ and $\gamma_{ji} = \overline{\gamma_{ij}}$ (for all $i,j \ge 0$, with $i+j \le 2k$. \ The {\it truncated complex moment problem} (TCMP) asks for necessary and sufficient conditions for the existence of a positive Borel measure $\mu$ supported in $\mathbb{C}$ such that $\gamma_{ij} = \int \overline{z}^i z^j \, d\mu$, $0 \leq i + j \leq 2k$. When such a measure exists, $\gamma$ is called a {\it truncated moment sequence} and $\mu$ is called a {\it representing measure} for $\gamma$. 

\par

The Riesz-Haviland Theorem \cite{Riesz23} tells us that a truncated sequence of complex numbers is a moment sequence if and only if the {\it Riesz functional} $\Lambda$ is nonnegative on nonnegative polynomials; that is, the functional $\Lambda: \mathbb{C}[z,\overline{z}] \to \mathbb{R}$, given by $\sum_{i,j} \alpha_{ij} z^i \overline{z}^j \mapsto \sum_{i,j} \alpha_{ij} \gamma_{ij}$, has the property that $\Lambda(p) \geq 0$ if and only if $p$ is a nonnegative polynomial. To conform with the prevailing usage in the literature, we will ordinarily refer to nonnegativity as positivity, and rephrase the above condition on the Riesz functional as requiring that $\Lambda$ be positive on positive polynomials. 

One way to detect the positivity of the Riesz functional is using the {\it moment matrix}, $M(k) \equiv M(k)(\gamma)$, defined as follows. First, let  
\begin{equation*}
                M[i,j] := 
                \begin{pmatrix}
                        \gamma_{i,j} & \gamma_{i+1,j-1} & \cdots & \gamma_{i+j,0} \\
                        \gamma_{i-1,j+1} & \gamma_{i,j} & \cdots & \gamma_{i+j - 1,1} \\
                        \vdots & \vdots & \ddots & \vdots \\
                        \gamma_{0,j+i} & \gamma_{1,j+i-1} & \cdots & \gamma_{j,i}
                \end{pmatrix} \\
\end{equation*}
Then put
\begin{equation*}
                M(k) :=
                \begin{pmatrix}
                        M[0,0] & M[0,1] & \cdots & M[0,k] \\
                        M[1,0] & M[1,1] & \cdots & M[1,k] \\
                        \vdots & \vdots & \ddots & \vdots \\
                        M[k,0] & M[k,1] & \cdots & M[k,k]
                \end{pmatrix}.
\end{equation*}
The moment matrix is positive semidefinite if and only if the Riesz functional is positive on sums of squares. 
That is to say, positivity of the moment matrix is a necessary condition for the existence of a representing measure. \ (For more on the basic theory of truncated moment problems, the reader is referred to \cite{CF91}, \cite{CF96},\cite{CF98}, \cite{CFM08} and \cite{Schmudgen17}.) 

\par

We connect the space of polynomials with the moment matrix using the so-called {\it functional calculus} for the moment matrix. \ 
Let $\mathcal{C} \equiv \mathcal{C}_{M(k)}$ be the column space of $M(k)$. \ 
We label each column in $M(k)$ as $1, Z, \overline{Z}, \overline{Z}^2, \overline{Z} Z, \overline{Z}^2, \ldots$, using the degree-lexicographic order. \ 
We then connect the columns to the space of polynomials using the map $\mathbb{C}[z,\overline{z}] \to \mathbb{C}$ given by
\begin{equation} \label{funccalc}
        \sum_{i,j} a_{ij} \overline{z}^i z^j \mapsto \sum_{i,j} a_{ij} \overline{Z}^i Z^j.
\end{equation}

\par 

It is important to note that $\overline{Z}^i Z^j$ (written in capital letters) is {\it not} a monomial in any polynomial ring, but rather a label for a vector in $\mathcal{C}$; monomials (and polynomials) will be represented using lowercase letters. 

\par
The fundamental ideas behind the existing results can be summarized as follows:

(a) if a moment problem has an associated moment matrix $M(n)(\gamma)$ that is singular, then at least one of its columns must be a linear combination of the others;

(b) each column relation can be described by a polynomial $p$;

(c) the zero set $\mathcal{Z}(p)$ of the above-mentioned polynomial must contain the support of a representing measure $\mu$, if such a measure exists;

(d) by taking the intersection of all possible zero sets (this is the so-called algebraic variety of $\gamma$, usually denoted by $V(\gamma)$), we obtain a first candidate for the support of $\mu$.

In fact, there is a canonical functional calculus that determines the {\it algebraic variety}, $\mathcal{V} \equiv \mathcal{V}(M(k))$, associated with the moment matrix: 
\begin{equation*}
        \mathcal{V} := \bigcap_{\substack{p(Z,\overline{Z}) = 0, \\ \deg{p} \leq k}} \mathcal{Z}(p),
\end{equation*}
where $\mathcal{Z}(p)$ denotes the zero set of $p(z,\overline{z})$. \ 
The moment matrix acts on the vector space of polynomials of degree $k$ by considering $p$ as a vector $\hat{p} \in \mathbb{C}^{\dim{\mathbb{C}[z,\overline{z}]_{k}}}$ whose entries are the coefficients of $p$, and multiplying by $M(k)$ on the left.

Observe then that $M(k) \cdot \hat{p} = 0$ if and only if $p(Z,\overline{Z}) \equiv 0$; that is, every $\hat{p}$ in the kernel of $M(k)$ corresponds to a column relation $p(Z,\bar{Z}) = 0$ of $M(k)$. \ 

Let $r := \text{rank}\,{M(k)}$ and $v := \text{card}\,{\mathcal{V}}$. \ Therefore, if $\gamma$ admits a representing measure, then 
\begin{equation*}
        r \le v.
\end{equation*}
When $r = v$, we will say that the moment problem is {\it extremal}. \ 

In the case when a representing measure $\mu$ exists, J. Stochel and F. Szafraniec proved in \cite{SS94} that the support of $\mu$ must be contained in $\mathcal{V}$, and therefore $r \leq \text{card}\,{\text{supp}\,{\mu} \leq v}$. \ As a result, if $p \in \mathbb{C}[z,\overline{z}]_{2k}$ and $p \vert_{\mathcal{V}} \equiv 0$, then 
\begin{equation} \label{consistency}
        \Lambda(p) = \int_{\mathbb{C}} p(z,\overline{z}) \, d\mu = 0.
\end{equation}
This is called the {\it consistency condition} for moment matrices. 

\par
We have thus discussed three necessary conditions for the existence of a representing measure for a given truncated complex sequence $\gamma$: (i) the moment matrix $M(k)$ must be positive semidefinite, (ii) the polynomials of degree at most $2k$ vanishing in $\mathcal{V}$ must satisfy \eqref{consistency}, and (iii) the rank of $M(k)$ must be bounded above by the cardinality of $\mathcal{V}$. \ It was shown by R.E. Curto, L.A. Fialkow, and M. M\"oller \cite{CFM08} that in the case of extremal moment problems, positivity and consistency are also {\it sufficient} for the existence of a representing measure.

In the sequel, we will consider in detail the extremal case. 

\subsection{Cubic column relations in the extremal TMP} \ In \cite{CY14}, the authors developed a research program to study the Sextic MP, that is, the case 
$M(3)\geq 0$, $M(2)>0$, with \textit{finite} algebraic variety. \ Since $M(3)$
is a square matrix of size $10$, the possible significant values for $\operatorname{rank} \; M(3)$ are $7$
and $8$. \ (While $\operatorname{rank} \; M(3)=9$ is theoretically possible, the associated algebraic variety would be infinite, being determined by a single column dependence relation.) \ When $r=v=7$, the approach in \cite{CY14} first focused on the case of a column relation given by a {\it harmonic} polynomial $q(z,\bar{z}):=f(z)-%
\overline{g(z)}$, where $f$ and $g$ are analytic polynomials, and $\deg \
q=3 $. \ Using degree-one transformations and symmetric properties of such
polynomials, the associated TMP can be reduced to the case $Z^{3}=itZ+u\bar{Z}$, with $%
t,u\in \mathbb{R}$. \ Wilmshurst \cite{WilmshurstThesis}, Crofoot-Sarason \cite{CrSa} and
Khavinson-{\v S}wiatek \cite{KhavSwia03} proved that for $\deg \ f=3$, we have $\operatorname{%
card}\ \mathcal{Z}(f(z)-\bar{z})\leq 7$. \ It then follows that a TMP
with this cubic column relation can have at most $7$ points in its algebraic variety. \ As proved in \cite{CY14}, the polynomial $q_{7}(z,\bar{z}):=z^{3}-itz-u\bar{z} ~(0<u<t<2u)$ has exactly $7$ zeros. \ In fact, $\mathcal{Z}(q_{7})$ consists of the origin, two points equidistant
from the origin (located on the bisector $z=i\bar{z}$), and four points on a
circle, symmetrically located with respect to the bisector (cf. Figure \ref{figurecubic}). \ There is also a cubic polynomial whose zero set is the union of the bisector and the circle, given by $q_{LC}(z,\overline{z}):=i(z-i \overline{z})(\overline{z}z-u)$. \ The fact that $\mathcal{Z}(q_7) \subseteq \mathcal{Z}(q_{LC})$ was crucial and led to the main result in \cite{CY14}.

\setlength{\unitlength}{1mm} \psset{unit=12mm} 

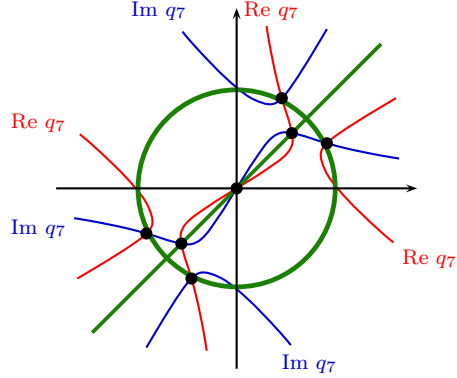
\begin{figure}[th]
\begin{center}

\begin{picture}(15,24)


\pscurve[linecolor=red](-.331,-1.8)(-.364,-1.6)(-.447,-1.2)(-.5,-1)(-.612,-.612)(-.618,-.5)(-.569,-.4)(-.315,-.2)(0,0)(.315,.2)(.569,.4)(.618,.5)(.612,.612)(.5,1)(.447,1.2)(.364,1.6)(.331,1.8)

\pscurve[linecolor=red](1.738,-.6)(1.618,-.5)(1.118,0)(.981,.2)(.935,.4)(1,.5)(1.14,.612)(1.766,1)

\pscurve[linecolor=red](-1.766,-1)(-1.14,-.612)(-1,-.5)(-.935,-.4)(-.981,-.2)(-1.118,0)(-1.618,.5)(-1.738,.6)


\pscurve[linecolor=blue](-1.8,-.331)(-1.6,-.364)(-1.2,-.447)(-1,-.5)(-.612,-.612)(-.5,-.618)(-.4,-.569)(-.2,-.315)(0,0)(.2,.315)(.4,.569)(.5,.618)(.612,.612)(1,.5)(1.2,.447)(1.6,.364)(1.8,.331)
\pscurve[linecolor=blue](-.6,1.738)(-.5,1.618)(0,1.118)(.2,.981)(.4,.935)(.5,1)(.612,1.14)(1,1.766)
\pscurve[linecolor=blue](-1,-1.766)(-.612,-1.14)(-.5,-1)(-.4,-.935)(-.2,-.981)(0,-1.118)(.5,-1.618)(.6,-1.738)


\pscircle[linecolor=darkgreen, linewidth=1.8pt](0,0){38.3pt}
\psline[linecolor=darkgreen, linewidth=1.4pt](-1.6,-1.6)(1.6,1.6)

\psset{linecolor=black}
\qdisk(1,0.5){2.3pt}
\qdisk(0.5,1){2.3pt}
\qdisk(-1,-0.5){2.3pt}
\qdisk(-0.5,-1){2.3pt}
\qdisk(0.612372,0.612372){2.3pt}
\qdisk(-0.612372,-0.612372){2.3pt}
\qdisk(0,0){2.3pt}

\put(-30,8){\color{red}{\footnotesize Re $q_7$}}
\put(-14,23){\color{blue}{\footnotesize Im $q_7$}}
\put(22,-10){\color{red}{\footnotesize Re $q_7$}}
\put(-30,-6){\color{blue}{\footnotesize Im $q_7$}}
\put(1,23){\color{red}{\footnotesize Re $q_7$}}
\put(6,-24){\color{blue}{\footnotesize Im $q_7$}}

\psline{->}(-2,0)(2,0)
\psline{->}(0,-2)(0,2)

\end{picture}

\end{center}

\vspace{75pt}

\caption{{\footnotesize The $7$-point set $\mathcal{Z}(q_{7})$. \ (The {\bf {\color{darkgreen} circle}} has radius $\sqrt{u}$.)}}
\label{figurecubic}

\end{figure}

\begin{thm}
\label{thmcubic} (\cite{CY14}) \ Let $M(3)\geq 0\,$, with $M(2)>0$ and $q_{7}(Z,\bar{Z})=0$.
\ Let $q_{LC}(z,\overline{z}):=i(z-i \overline{z})(\overline{z}z-u)$. \ 
Then there exists a representing measure for $M(3)$ if and only if

\begin{equation*}
\left\{ 
\begin{array}{ccc}
\Lambda (q_{LC}) & = & 0 \\ 
\Lambda (zq_{LC}) & = & 0. 
\end{array}%
\right.
\end{equation*}

\end{thm}

Since the associated TMP is extremal, what was needed to establish Theorem \ref%
{thmcubic} was a verification of consistency. \ To this end, \cite{CY14} proved
a result on representations of polynomials, as follows.

\begin{prop} \ (\textrm{Representation of Polynomials}) \ (\cite{CY14}) \
Let $\mathcal{Q}_{6}:=\{p\in \mathbb{C}_{6}[z,\bar{z}]:p|_{\mathcal{Z}%
(q_{7})}\equiv 0\}$ and let $\mathcal{I}:=\{p\in \mathbb{C}_{6}[z,\bar{z}%
]:p=fq_{7}+g\bar{q}_{7}+hq_{LC}$ for some $f,g,h\in \mathbb{C}_{3}[z,\bar{z}%
]\}$. \ Then $\mathcal{Q}_{6}=\mathcal{I}$.
\end{prop}

Using the above approach as a guiding principle, R.E. Curto and S. Yoo went on to complete the analysis of the extremal case for the sextic MP \cite{CY15}; this required an application of the Division Algorithm.

While R. Curto and S. Yoo focused on a rather narrow set of harmonic polynomials, our aim in this paper is to cover large swaths of polynomials in $\mathbb{C}[z,\overline{z}]$ of the form 
\begin{equation*}
    p(z,\overline{z}) = f(z) + \overline{g(z)},
\end{equation*}
that is, {\it the entire class of harmonic polynomials}, in the sense that their Laplacian relative to $z$ and $\overline{z}$ is identically zero. 

\subsection{Commutative Algebra}

In order to define a Gr\"obner basis, we first define a {\it monomial ordering}. 
\begin{defn}
Let $R = \mathbb{K}[x_1,\ldots,x_n]$ be a polynomial ring over a field $\mathbb{K}$. \ 
A monomial ordering on $R$ is a total ordering on the monomials of $R$. 
\end{defn}
We denote by $\mathrm{LM}{(p)}$ the {\it leading} monomial (or {\it initial} term) of the polynomial $p$ with respect to a monomial ordering $>$.
Given an ideal $I \subseteq R$, denote by $\mathrm{LM}{(I)}$ the {\it initial monomial ideal}; i.e., this is the ideal generated by all of the leading monomials of elements of $I$. 
For an arbitrary set $S \subseteq R$, we denote by $\mathrm{LM}{(S)}$ the leading monomial ideal of the ideal generated by $S$. 
\begin{defn}
Let $R = \mathbb{K}[x_1,\ldots,x_n]$ be a polynomial ring over an algebraically closed field $\mathbb{K}$, and fix a monomial ordering, $>$, on $R$. \ Let $I$ be an ideal in $R$ and let $G$ be a finite subset of $R$. \ If $\mathrm{LM}{(I)} = \mathrm{LM}{(G)}$, then $G$ is called a {\it Gr\"obner basis} for $I$. \ A Gr\"obner basis is said to be {\it minimal} if the leading monomial of every element of $G$ is irreducible by the other elements of $G$. 
\end{defn}
Note that in general, Gr\"obner bases are not unique. 
However, if we fix a monomial ordering on our polynomial ring, then {\it minimal} Gr\"obner bases are unique. 

Throughout this paper, when we refer to a Gr\"obner basis, we will mean a minimal Gr\"obner basis with respect to the degree lexicographic order for our polynomial ring. 

\par

The well-known Hilbert's Nullstellensatz, as well as the weak Nullstellensatz, are two key tools to the proof of one of our main results. 
\begin{thm}
        (Weak Nullstellensatz) Let $\mathbb{K}$ be algebraically closed and let $I \subseteq \mathbb{K}[x_1,\ldots, x_n]$ be an ideal. \ Define
        \begin{equation*}
                \mathbf{V}(I) := \{ a \in \mathbb{K}^n \, : \, f(a) = 0 \quad \textrm{ for all } f \in I \}.
        \end{equation*}
        If $\mathbf{V}(I) = \emptyset$, then $I = \mathbb{K}[x_1,\ldots, x_n]$. 
\end{thm}

\begin{thm}
        (Hilbert's Nullstellensatz)  \ Let $\mathbb{K}$ be algebraically closed, and let $V$ be an algebraic subset of $\mathbb{K}^n$. \ Define $\mathbf{I}(V) := \{ f \in \mathbb{K}[x_1,\ldots,x_n] \, : \, f(a) = 0 \quad \textrm{ for all } a \in V \}$. \ For $f_1,\ldots, f_s \in \mathbb{K}[x_1,\ldots,x_n]$, let $(f_1,\ldots,f_s)$ denote the ideal generated by $f_1,\ldots,f_s$. \ Then $f \in \mathbf{I}(\mathbf{V}((f_1,\ldots,f_s)))$ if and only if $f^r \in (f_1,\ldots,f_s)$ for some integer $r \geq 1$. 
\end{thm}
There are two important consequences of the above theorems. \ First, every ideal in a polynomial ring corresponds to an algebraic variety. \ Second, $\mathbf{I}(\mathbf{V}(I)) = \sqrt{I}$ for every ideal $I$ of a polynomial ring. 

\par

The following theorem gives a canonical basis for the quotient ring of a polynomial ring by an ideal of that ring. 
\begin{thm}[\cite{Eisenbud95}, Theorem 15.3] 
        Let $R = \mathbb{K}[x_1,\ldots,x_n]$ and $I \subseteq R$ be any ideal. \ 
        Then for any monomial order $>$ on $R$, the set $B$ of all monomials not in $\mathrm{LM}{(I)}$ forms a basis for $R/I$. 
\end{thm}
The following theorem and proposition relate the size of a finite variety with the quotient of a polynomial ring by an ideal. 
\begin{thm}[\cite{CLO15} Theorem 5.3.6]
    Suppose that $R = \mathbb{K}[x_1,\ldots,x_n]$ and $I \subseteq R$ is an ideal. \ Then the following are equivalent.
    \begin{enumerate}
            \item{
                            For every $i = 1, \ldots, n$, there exists some integer $\alpha_i$ so that $x_{i}^{\alpha_i} \in \mathrm{LM}{(I)}$.
                    }
            \item{
                            The dimension of $R/I$ as a $\mathbb{K}$-vector space is finite.
                    }
            \item{
                            $\mathscr{V} = \mathscr{V}(I)$ is a finite set.
                    }
    \end{enumerate}
\end{thm}

\begin{prop}[\cite{CLO15} Proposition 5.3.7]
        Let $R$, $I$, and $\mathscr{V} = \mathscr{V}(I)$ be as before, and suppose that for each $i = 1, \ldots, n$, we have $x_i^{\alpha_i} \in \mathrm{LM}{(I)}$, where $\alpha_i \in \mathbb{N}_0$. \ Then:
        \begin{enumerate}
                \item{
                                The number of points of $\mathscr{V}$ is at most $\dim_{\mathbb{K}}{R/I}$. 
                        }
                \item{
                                If $I$ is radical and $\mathbb{K}$ is algebraically closed, then equality holds in the previous part. \ That is, the number of points in $\mathscr{V}$ is exactly $\dim_{\mathbb{K}}{R/I}$. \ Put another way, the number of points in $\mathscr{V}$ is equal to the number of monomials {\it not} divisible by any leading monomial in $G$.
                        }
        \end{enumerate}
\end{prop}

\section{Preliminary Results}

In order to apply many of the tools from commutative algebra, it is often useful to consider $z$ and $\overline{z}$ as ``formally independent'' from one another. \ 
That is, rather than considering $p \in \C[z,\overline{z}]$ as a function of a single variable $z \in \C$, we consider $p$ to be an element of $\C[z,w]$, i.e., a function of two variables, $(z,w) \in \C^2$. \ 
This can be summarized by the following commutative diagram:
\begin{figure}[ht]
\centering
\begin{tikzcd}
	{\mathbb{C}[z,\overline{z}]} && {\mathbb{C}[z,w]} \\
	\\
	&& {\mathbb{C}[z,w]/\sim}
	\arrow["{\varphi}", from=1-1, to=1-3]
	\arrow[from=1-1, to=3-3]
	\arrow["{\pi}", from=1-3, to=3-3]
\end{tikzcd}
\end{figure}

\noindent where $w \sim \overline{z}$. The map $\varphi : \C[z,\overline{z}] \to \C[z,w]$ is the algebra homomorphism that sends $z \mapsto z$ and $\overline{z} \mapsto w$, and $\pi: \C[z,w] \to \C[z,w]/\sim$ is the map that sends $z \mapsto z$ and $w\mapsto \overline{z}$.

\par

We can now make precise the idea of ``considering $z$ and $\overline{z}$ as formally independent.'' \ 
For $p \in \C[z,\overline{z}]$, consider the image of $p$ under $\varphi$; this will allow us to use the tools of commutative algebra discussed previously. \ 
We denote by $\overline{p}(z,w)$ the image of $\overline{p(z,\overline{z})}$ under $\varphi$. \ 
Additionally, it makes sense to consider the image of $\mathcal{V}$ under the map $\psi : \C \to \C^2$ sending $z \in \C$ to $(z,\overline{z}) \in \C^2$. \ 
Now put $\mathscr{V}:= \psi(\mathcal{V})$, and let us employ the tools from commutative algebra on polynomials in two indeterminates. \ 
To obtain the desired moment results, we take the quotient by the relation $w \sim \overline{z}$.

\par

A crucial result in moment theory was obtained by J. Stochel and F. Szafraniec in 1994. 

\begin{thm}{(\cite{SS94}, Proposition 1)}
        Let $p \in \mathbb{C}[z,\overline{z}]$ and let $\gamma:\mathbb{N}^2 \to \mathbb{C}$ be a complex moment sequence with representing measure $\mu$. \ 
        Then $\gamma$ is positive semidefinite and the following conditions are equivalent:
        \begin{enumerate}
            \item{the (closed) support of $\mu$ is contained in $\mathcal{Z}(p)$;}
            \item{$\Lambda(p) = 0$.}
        \end{enumerate}
\end{thm}

\par

To state one of our main results, we first need to briefly describe the main result in \cite{CY14}; there, the authors focused on a particular family of harmonic polynomials, as follows. \ For $u,t \in \mathbb{R}$, put
\begin{align*}
        q_{7}(z,\overline{z}) &:= z^3 - i t z - u \overline{z} \\
        q_{LC}(z,\overline{z}) &:= i (z - i \overline{z})(\overline{z}z - u).
\end{align*}

\begin{thm}{(\cite{CY14}, Theorem 3.1)} \label{cy14}
        Let $M(3) \geq 0$, with $M(2) > 0$. \ 
        Let $q_7(Z,\overline{Z}) \equiv 0$ with $0 < |u| < t < 2|u|$. \ 
        With $q_{LC}$ defined as before, the following statements are equivalent:
        \begin{enumerate}
                \item{there exists a representing measure for $M(3)$;}
                \item{
                                $
                                \begin{cases}
                                        \varLambda(q_{LC}) = 0, \\
                                        \varLambda(zq_{LC}) = 0;
                                \end{cases}
                                $
                        }
                \item{
                                $
                                \begin{cases}
                                        \Re{(\gamma_{12})} - \Im{(\gamma_{12})} = u ( \Re{(\gamma_{01})} - \Im{(\gamma_{01})}), \\
                                        \gamma_{22} = (t+u) \gamma_{11} - 2u \Im{(\gamma_{02})};
                                \end{cases}
                                $
                        }
                \item{
                                $q_{LC}(Z,\overline{Z}) \equiv 0.$
                        }
        \end{enumerate}
\end{thm}

In the next section, we will describe our main results; as a starting point, we need to prove a generalization of \cite[Theorem 3.1]{CY14}.

\section{Main Results}

In this section, we state and prove the main results of this paper. \ First, we prove a generalization of \cite{CY14} (Theorem 3.1). \ 
In order to establish Theorem \ref{mainthm}, we first state and prove two key results. \ Toward this goal, we need to discuss a motivating example.

Consider a member of the polynomial family studied in \cite{CY14} (with suitable coefficients chosen), $q_7(z,\overline{z}) = z^3 - 8iz + 5 \overline{z}$.
This polynomial has seven roots, 
\begin{equation*}
    \mathcal{Z}{(q_7)} = \left\{ 0, \pm(1 + 2i), \pm(2 + i), \pm\sqrt{\frac{3}{2}} (1 + i) \right\}.
\end{equation*}
Next, consider the image of $q_7$ under the map $\varphi$ defined above, that is, $q_7(z,w) = z^3 - 8i z + 5w$.\ Now let
\begin{equation*}
    I := \left\{ p \in \mathbb{C}[z,w] \, : \, p \vert_{\mathscr{V}} \equiv 0 \right\}.
\end{equation*}
If we compute the minimal Gr\"obner basis $G$ with respect to the degree lexicographic order on $\mathbb{C}[z,w]$ for $I$, we find that $G$ contains $q_7$, $\overline{q_7}$, {\it and} the polynomial which was denoted $q_{LC}$ in \cite{CY14}, 
\begin{equation*}
    q_{LC}(z,w) = i(z - i w)(w z - u)
\end{equation*}
This example inspires the following theorem.

\begin{thm} \label{thm31}
Suppose $M(k) \geq 0$, with a column relation $p(Z,\overline{Z}) = 0$ given by a harmonic polynomial $p \in \mathbb{C}[z,\overline{z}]_{k}$. \ Let $\mathcal{V}$ be the algebraic variety associated with $M(k)$ and let $G$ the Gr\"obner basis for the ideal associated to $\mathcal{V}$ (considered as a subset of $\mathbb{C}^2$). \ Then the cardinality of $G$ is equal to the nullity of $M(k)$. \ In particular, the elements of $G$ correspond to the column relations of $M(k)$. 
\end{thm}

\begin{proof}
Let $I = \mathbf{I}(\mathcal{V})$ be the ideal associated with $\mathcal{V}$. \ 
Since $\mathcal{V}$ is finite, we must have pure powers of $z$ and $w$ in $\mathrm{LM}(I)$, meaning that our Gr\"obner basis must include polynomials with pure powers of $z$ and $w$. \ 
Since we have $p(Z,\overline{Z}) \equiv 0$, we also have $\overline{p(Z,\overline{Z})} \equiv 0$.
Since $p$ and $\overline{p}$ both vanish on $\mathcal{V}$, they are elements of $I$. 

\par

Next, we claim that $p$ and $\overline{p}$ must have different leading monomials, one of which is a pure power of $z$ and the other a pure power of $w$. \ 
That they are pure powers comes from the fact that $p(z,\overline{z})$ is harmonic and thus of the form $p(z,\overline{z}) = f(z) + \overline{g(z)}$. \ 
Therefore, $p(z,w) = f(z) + g(w)$, so there are no monomials containing both $z$ and $w$. \ 

Now we prove that $\mathrm{LM}(p)$ and $\mathrm{LM}(\overline{p})$ are pure powers of different indeterminates. \  
Suppose otherwise. \ 
Then $p$ would contain the leading monomial $z^n$ as well as the monomial $w^n$. \ 
If this were the case, then $\overline{p}$ would also have $z^n$ as its leading monomial as well as $w^n$ as another monomial. \ 
However, we know that the number of roots of $p$ must be finite, which can only happen if $\mathcal{Z}(p,\overline{p})$ is finite, meaning that $p$ and $\overline{p}$ each have a different leading monomial, one a pure power of $z$ and the other a pure power of $w$, as shown in \cite[Theorem 5.3.6]{CLO15}.

\par

Thus, in our Gr\"obner basis we will have $p(z,w)$ and $\overline{p}(z,w)$, since these are both polynomials vanishing on the set and whose leading monomials are not divisible by the leading monomials of any other elements of $I$, which, without loss of generality, both have $1$ as their leading coefficient. \ 
We know that for a polynomial $q \in \C[z,w]_k$, $\hat{q} \in \ker{M(k)}$ if and only if $q(Z,\overline{Z}) = 0$ if and only if 
\begin{equation*}
         q\vert_{\mathscr{V}} \equiv 0.
\end{equation*}
Therefore, the remaining column relations are all elements of the Gr\"obner basis as well, since, without loss of generality, they will also have $1$ as their leading coefficient and their leading terms cannot be divisible by any other leading terms. \ 

Indeed, suppose we have $q \in \mathbb{C}[z,w]$, $\tilde{q} \in G$ such that $\mathrm{LM}{(q)}$ is divisible by $\mathrm{LM}{(\tilde{q})}$. 
Then $q$ would have to define a column relation for $M(k)$, which would contradict the assumption that our sequence was extremal. 

\end{proof}

\begin{thm} \label{thm32}
Suppose $M(k) \geq 0$ with column relation $p(Z,\overline{Z}) = 0$ given by $p \in \mathbb{C}[z,\overline{z}]_{k}$ harmonic. \ 
Let $\mathcal{V}$ be the variety associated with $M(k)$ and $G$ the Gr\"obner basis for the ideal associated with $\mathscr{V}$ (considered as a subset of $\mathbb{C}^2$). \ 
        Define the sets $\mathscr{P}_{2k}$ and $\mathscr{I}_{2k}$ as follows:
        \begin{equation}
                        \mathscr{P}_{2k} := \left\{ p \in \mathbb{C}[z,w]_{2k} \, : \, p_{\mathcal{V}} \equiv 0 \right\}
        \end{equation}
and
        \begin{equation}
                        \mathscr{I}_{2k} := \left\{ p \in \mathbb{C}[z,w]_{2k} \, : \, p = \sum_{j=1}^{v} a_{j} g_{j} , \, a_j \in \mathbb{C}[z,w]_{k} \right\}.
        \end{equation}
        Then $\mathscr{P}_{2k} = \mathscr{I}_{2k}$.
\end{thm}

\begin{proof}
        Let $v = \card{\mathcal{V}}$. \ 
        Define $T : \left( \mathbb{C[}z,w]_{2k}\right)^{v} \to \C[z,w]_{2k}$ by
        \begin{equation*}
                T: (a_1,\ldots,a_v) \mapsto \sum_{j=1}^{v} a_j g_j, \quad a_j \in \C[z,w]_{k}
        \end{equation*}
        and $S: \C[z,w]_{2k} \to \C^v$ by
        \begin{equation*}
                S : p \mapsto (p(z_1,w_1), \ldots, p(z_v,w_v)), \quad (z_j,w_j) \in \mathcal{V} \subseteq \C^2.
        \end{equation*}
        It is clear that the image of $T$ is equal to $\mathscr{I}_{2k}$ and the kernel of $S$ is equal to $\mathscr{P}_{2k}$. \ 
        It is also clear that $\mathscr{I}_{2k} \subseteq \mathscr{P}_{2k}$. \ 
        We will prove that $\mathscr{P}_{2k} \subseteq \mathscr{I}_{2k}$. \ 
        Suppose that $p \in \ker{S}/\Ran{T}$, so $p = b + \Ran{T}$, with $b \in \ker{S}$. \ 
        Now let $I$ denote the ideal in $\C[z,w]$ generated by $G$ (with respect to the degree lexicographic order). \ 
        We know that the number of irreducible monomials in $\mathrm{LM}\,(I)$ is equal to the number of roots shared by the elements of $I$, which in our case is $v$. \ 
        Since we assume that the rank of $M(k)$ is equal to $v$, this is also the number of linearly independent basis elements in the column space of $M(k)$. \ 
        By construction, the irreducible monomials are the ones corresponding to the linearly independent columns. \ 
        This means that if $b$ is not reducible by $G$, its leading monomial is not divisible by any of the monomials in $\mathrm{LM}\,(I)$. \ 
        Since $b \in \ker{S} = \mathscr{P}_{2k}$, it vanishes on $\mathcal{V}$ and $\deg{b} \leq k$, $M(k) \cdot \hat{b} = 0$. \ 
        Therefore if $b$ were nonzero, it would correspond to another column relation, contradicting our assumption of extremality. \ 
        Thus, $b = 0$ and we have that $\mathscr{P}_{2k} \subseteq \mathscr{I}_{2k}$, proving the lemma. 
        
\end{proof}

The following result gives a concrete solution to the extremal moment problem for arbitrary $M(k)$ with a column relation $p(Z,\bar{Z})=0$, in terms of the variety associated to $p$ and $\bar{p}$; observe that no assumption on the harmonicity of $p$ is made.

\begin{thm} \label{mainthm}
        Let $k \ge 1$ and let $M(k) \geq 0$ be the moment matrix associated with a finite sequence of complex numbers $\gamma$. \ Suppose that  $p(Z,\overline{Z}) \equiv 0$ for some polynomial $p \in \mathbb{C}[z,\overline{z}]_{k}$. \ 
        Let $\mathcal{V}$ be the variety associated with $p$ and $\overline{p}$,and assume that $r:= \rk{M(k)}$ is equal to $v:= \card{\mathcal{V}}$. \ 
        Let $G = \{g_i\} \subseteq \mathbb{C}[z,w]$ be the Gr\"obner basis for the ideal associated with $\mathscr{V}$ (taken as a subset of $\mathbb{C}^2$). \ 
        Then the following statements are equivalent:
        \begin{enumerate}
                \item{there exists a representing measure for $\gamma$;}
                \item{$\Lambda(g_i(z,\overline{z}) = 0$ and $\Lambda(z \, g_i(z,\overline{z})) = 0$ for all $1 \leq i \leq \card{G}$;}
                \item{$g_i(Z,\overline{Z}) \equiv 0$ for all $1 \leq i \leq \card{G}$.}
        \end{enumerate}
\end{thm}

\begin{proof}
    We first prove the equivalence of (1) and (2). 
    
    (1) $\Rightarrow$ (2): This is rather straightforward. \ For, if $\mu$ is a representing measure for $\gamma$, we know that $3k-2 \leq r \leq \text{card}\,{\text{supp}\,{\mu}} \leq 3k-2 = v$,
    and therefore $\text{supp}\,{\mu} = \mathcal{V}$ and $r = v$.
    Thus,
    \begin{equation*}
        \varLambda(g_i) = \int_{\mathbb{C}} g_i \, d\mu = 0.
    \end{equation*}
    Since $\text{supp}\,{\mu} \subseteq \mathcal{V}$, we also have $\varLambda(z g_i) = 0$.

    \medskip

    (2) $\Rightarrow$ (1): From Theorem \ref{thm31}, we know that for every $g_i \in G$, $g_i(Z,\overline{Z}) \equiv 0$ is a column relation for $M(k)$. \ Since $g_i$ vanishes on $\mathcal{V}$, $mg_i$ vanishes on $\mathcal{V}$ for all monomials $m$ and, in particular, for all monomials of degree at most $k$. \ Therefore, since $mg_i \in \mathbb{C}_{2k}[z,w]$ vanishes on $\mathcal{V}$, Theorem \ref{thm32} tells us that every $mg_i \in \mathcal{P}_{2k} = \mathscr{I}$. \ 
    In other words, we have 
    \begin{equation*}
        mg_i = \sum_{j=0}^{\text{card}\,{G}} a_j g_j \quad \text{ for } \quad a_j \in \mathbb{C}_{k}[z,w].
    \end{equation*}
    Thus, $mg_i$ is annihilated by the Riesz functional for all monomials $m$ with $\deg{m} \leq k$. \ Therefore, consistency has been verified.

    \medskip

    Next, we prove the equivalence of (2) and (3). 
    
(2) $\Rightarrow$ (3): If $g_i(Z,\overline{Z}) \equiv 0$, then we know that $g_i \in \ker{\Lambda}$. \ This also means that $\Lambda(q(z,\overline{z}) \, g_i(z,\overline{z})) = 0$ for all $q \in \mathbb{C}[z,\overline{z}]_{k}$, so in particular $\Lambda{(z \, g_i(z,\overline{z}))} = 0$. 

(3) $\Rightarrow$ (2): Observe that $g_i \in \ker{\Lambda}$ if and only if $g_i(Z,\overline{Z}) \equiv 0$. 
\end{proof}

\subsection{Numerical Results}

Notably absent from Theorem \ref{mainthm} is the numerical condition that was present in Theorem \ref{thmcubic}, proved in \cite{CY14}. Such condition can be determined on a problem-by-problem basis once the Gr\"obner basis has been computed. If we suppose first that a representing measure does exist, then by evaluating the polynomial under the Riesz functional (condition 2 in Theorem \ref{mainthm}), we obtain equations that give the numerical conditions which are equivalent to the existence of a representing measure. 
If $g_i = \sum_{i,j} a_{ij} \overline{z}^i z^j$, then the numerical conditions are as follows:
\begin{equation}
        \sum_{i,j} a_{ij} \gamma_{ij} = 0
\end{equation}
and
\begin{equation}
        \sum_{i,j} a_{ij} \gamma_{i,j+1} = 0,
\end{equation}
using Theorem \ref{mainthm}(2).

\section{Examples}

\subsection{Reproducing the main result of Curto and Yoo}

In this section, we illustrate Theorem \ref{mainthm} using concrete examples.  
First, we are able to reproduce the main result in \cite{CY14}. 
For example, choosing $u = 5$ and $t=8$, we find that the Gr\"obner basis associated with the roots of $q_7(z,\overline{z}) = z^3 - i 8 z - 5 \overline{z}$ consists of $q_7$, $\overline{q_7}$ and $q_{LC}$.

Below are plots from the computation for $u=5, t=8$. Figure \ref{cyq7} shows the plot of $q_7$ and Figure \ref{cylc} shows the plot of the column relation found by computing the Gr\"obner basis.

\begin{figure}[H] 
    \centering
    \includegraphics[width=0.6\textwidth]{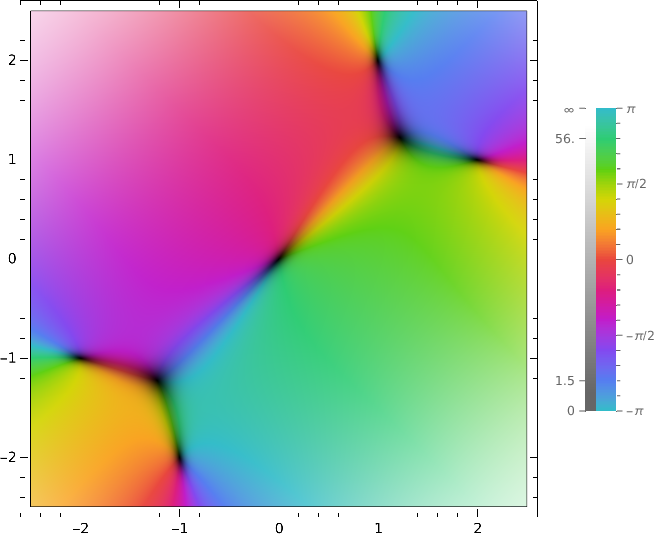}
    \caption{Complex plot of $q_7(z,\overline{z}) = z^3 - i 8 z + 5 \overline{z}$. \ We can clearly see the seven points of $\mathcal{V}$ in the plot of $q_7$.}
    \label{cyq7}
\end{figure}
        
\begin{figure}[H] 
        \begin{center}
        \includegraphics[width=0.6\textwidth]{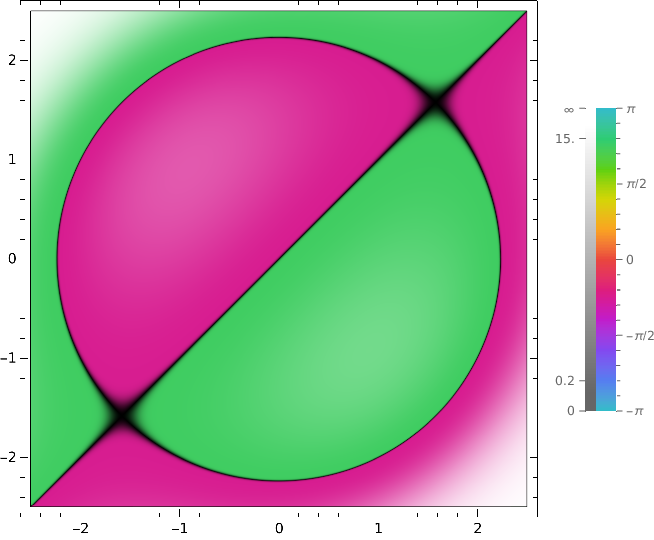}
                \caption{Complex contour plot of the polynomial found in the Gr\"obner basis, which is exactly $q_{LC}(z,\overline{z}) = i(z - i\overline{z})(\overline{z}z - 5)$.}
                \label{cylc}
        \end{center}
\end{figure}
We can use the numerical methods described above to find the same condition on the moments found in \cite{CY14}. 

\subsection{A different cubic result}

Next, we look at an example using the polynomial $\overline{z} + 3z^2 + 2z^3$, described in \cite{Wilmshurst98}.  
Suppose $M(3) \geq 0$ with $M(2) > 0$, and suppose that $M(3)$ has a column relation corresponding to the polynomial $p_{w3}(z,\overline{z}) := \overline{z} + 3z^2 + 2z^3$. 
We now consider the polynomial $p_{w3} \in \mathbb{C}[z,w]$ to be a polynomial of two independent variables, $(z,w) \in \mathbb{C}^2$. 
The Gr\"obner basis associated with $p_{w3}$ and $\overline{p_{w3}}$ constructed following the procedure in \cite{White25} contains the following elements: \vspace{-10pt}
        \begin{equation*}
                \begin{aligned}
                        g_1(z,w) &= w + 3z^2 + 2z^3 \\
                        g_2(z,w) &= -z + w + 2z^2 - 2w^2 + 4z^2 w - 4zw^2 \\
                        g_3(z,w) &= z + 3w^2 + 2w^3. \\
                \end{aligned}
        \end{equation*}
Thus, we know that the third column relation must be $g_2(Z,\overline{Z}) \equiv 0$. \ We again present plots of the polynomials corresponding to the column relations of $M(3)$ (cf. Figure \ref{fig4}). 
        \begin{figure}[H] 
                \begin{center}
                \includegraphics[width=0.49\textwidth]{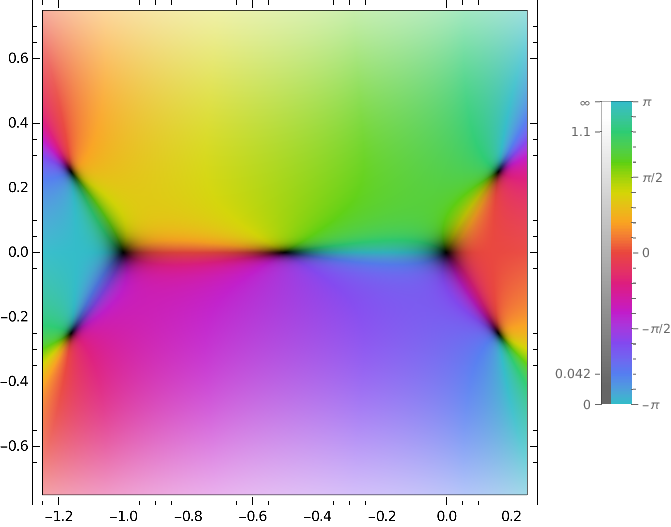}
                \includegraphics[width=0.49\textwidth]{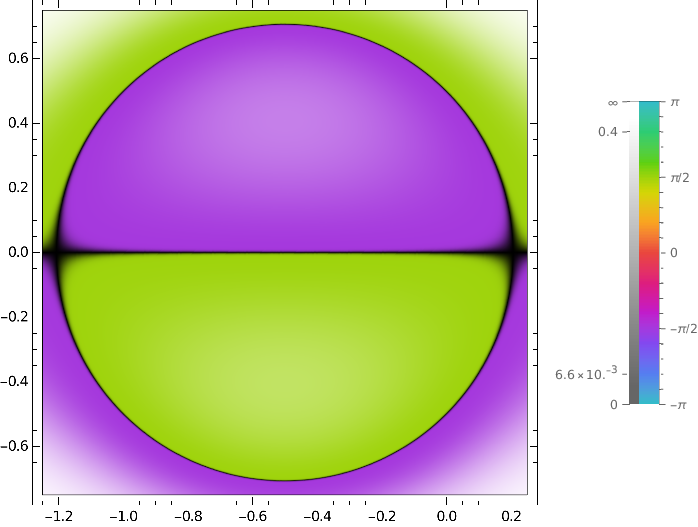}
                \caption{Complex contour plot of $g_1(z,\overline{z})$ (left) and of the polynomial found in the Gr\"obner basis, $g_2(z,\overline{z})$ (right)}
                \label{fig4}
                \end{center} 
        \end{figure}
We can then find the concrete condition to guarantee the existence of a representing measure for $M(3)$; that is, 
\begin{equation*}
        \begin{aligned}
                0 = \varLambda{(g_2(z,\overline{z}))} &= 2i \cdot  \Im{(\gamma_{10})} - 4i \cdot \Im{(\gamma_{20})} - 8i \cdot  \Im{(\gamma_{21})} ; \\
                0 = \varLambda{(z g_2(z,\overline{z}))} &= -\gamma_{02} + \gamma_{11} + 2 \gamma_{03} - 2 \gamma_{21} + 4 \gamma_{13} - 4 \gamma_{22}.
        \end{aligned}
\end{equation*}

\subsection{A quartic example}

Next, consider the degree-four polynomial also given in \cite{Wilmshurst98}:
\begin{equation*}
        p(z,\overline{z}) = z^4 + \frac{4}{3} z^3 + 2 z^2 + 3\overline{z} + 4z + \frac{11}{5}.
\end{equation*}
Observe that this polynomial has ten roots of the form $(z,\overline{z})$, so we expect to find a Gr\"obner basis with five polynomials all with degree at most four. \ 
Indeed, this is the case. \ 
Unfortunately, finding the roots of $p(z,w)$, and moreover finding the coefficients for the polynomials in our Gr\"obner basis, cannot be done nicely. 
By letting $z = x+iy$ and taking the resultant $\text{Res}(\Re{p(z,\overline{z})},\Im{p(z,\overline{z})},y)$, we find that in order to find the roots of $p(z,\overline{z})$, we need to find the roots of 
\begin{equation*}
    -303-2688x-6672x^2 + 1120 x^3 + 20160 x^4 + 17280 x^5 + 8640 x^6,
\end{equation*}
whose Galois group is the symmetric group on six elements. 
Nonetheless, Theorem \ref{mainthm} says that the Gr\"obner basis must have five elements, all corresponding to the column relations of $M(4)$. \ The plots of the relevant polynomials are shown below in Figure \ref{g1g2} and Figure \ref{g3}. 
\begin{figure}[H]
        \begin{center}
        \includegraphics[width=0.49\textwidth]{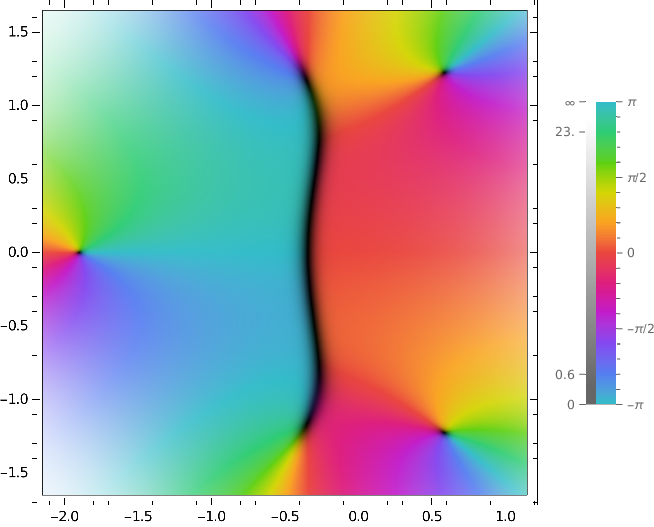}
        \includegraphics[width=0.49\textwidth]{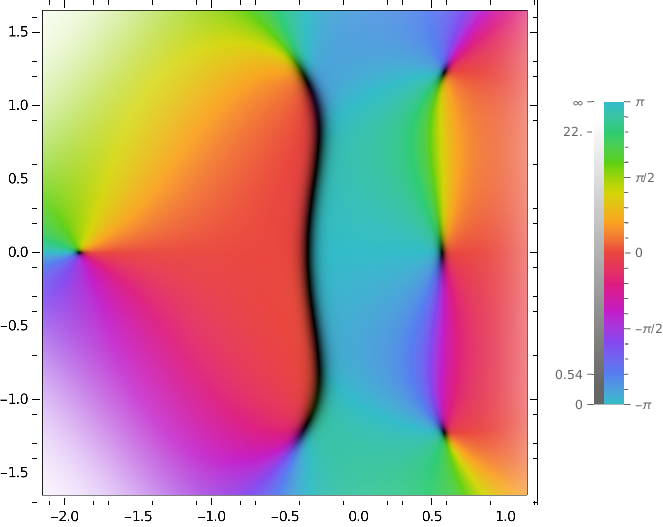}
        \caption{Complex contour plot of $g_1(z,\overline{z})$ (left) and $g_2(z,\overline{z})$ (right).}
        \label{g1g2}
        \end{center}
\end{figure}
\begin{figure}[H]
        \begin{center}
        \includegraphics[width=0.6\textwidth]{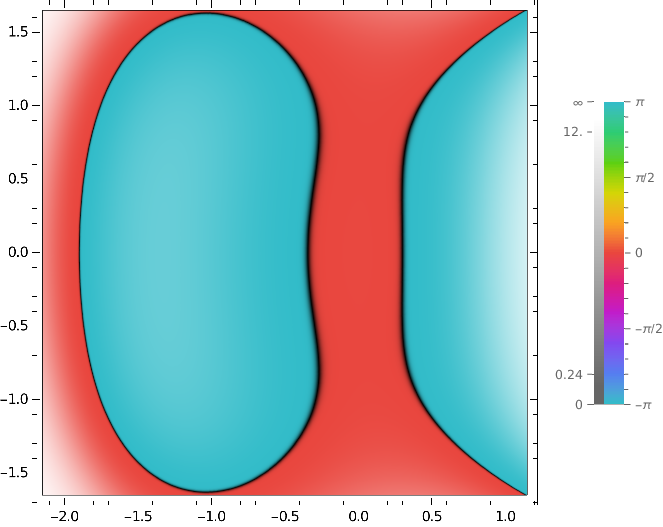}
                \caption{Complex contour plot of $g_3(z,\overline{z})$}
                \label{g3}
        \end{center}
\end{figure}

\section{Acknowledgments}
The first-named author was partially supported by U.S. NSF grant DMS-2247167. \ Several examples in this paper were obtained using calculations with the software tool \textit{Mathematica} \cite{Mathematica}.

\section{Declarations}

\subsection{Conflicts of interest/competing interests} \ None.

\subsection{Non-financial interests} \ The first-named author is on the Editorial Board of
{\it Journal of Mathematical Analysis and Applications}.

\subsection{Data availability} \ All data generated or analyzed during this study are included in this article.

\end{document}